\documentclass[12pt]{amsart}
\usepackage{tikz-cd} 
\usepackage[english]{babel}
\usepackage[utf8x]{inputenc}
\usepackage[T1]{fontenc}
\usepackage{amsmath}
\usepackage{amstext}
\usepackage{amsfonts}
\usepackage{amssymb}
\usepackage{amsthm}
\usepackage{mathtools}
\usepackage{dsfont}
\usepackage{color}
\usepackage{graphicx}
\usepackage[colorinlistoftodos]{todonotes}
\usepackage[colorlinks=true, allcolors=blue]{hyperref}
\usepackage{float}
\usepackage[colorinlistoftodos]{todonotes}
\usepackage{theoremref}
\usepackage{enumitem}
\usepackage{dirtytalk}
\usepackage{todonotes}
\usepackage{marginnote} 

\usepackage{multicol}  
\allowdisplaybreaks
\let\oldaddcontentsline\addcontentsline
\newcommand{\stoptocentries}{\renewcommand{\addcontentsline}[3]{}}
\newcommand{\starttocentries}{\let\addcontentsline\oldaddcontentsline}
\newtheorem{theorem}{Theorem}[section]
\newtheorem{lemma}[theorem]{Lemma}
\newtheorem{prop}[theorem]{Proposition}

\newtheorem*{cor*}{Corollary}
\newtheorem*{conjecture*}{Conjecture}
\newtheorem*{thm*}{Theorem}
\newtheorem*{lem*}{Lemma}

\newtheorem*{prop*}{Proposition}
\theoremstyle{definition}
\newtheorem{definition}[theorem]{Definition}

\newtheorem{example}[theorem]{Example}
\newtheorem*{defn*}{Definition}

\theoremstyle{remark}
\newtheorem{remark}[theorem]{Remark}

\newcommand{\M}{\mathcal{M}}

\newcommand{\cH}{\mathcal{H}}

\newcommand{\G}{\Gamma}

\title{Subalgebras, subgroups and singularity}

\date{\today}

\begin{document}
\author{Tattwamasi Amrutam}
\address{T. Amrutam, Department of mathematics, Ben Gurion University of the Negev,
P.O.B. 653, Be'er Sheva, 8410501, Israel}
\email{tattwama@post.bgu.ac.il}
\author{Yair Hartman}
\address{Y. Hartman, Department of mathematics, Ben Gurion University of the Negev,
P.O.B. 653, Be'er Sheva, 8410501, Israel}
\email{hartmany@bgu.ac.il}
\thanks{This research was supported by ISRAEL SCIENCE FOUNDATION grant 1175/18. }
\maketitle

\begin{abstract}
This paper is concerned with the non commutative analog of the Normal Subgroup Theorem for certain groups. Inspired by~\cite{KalPan}, we show that all $\Gamma$-invariant subalgebras of $L\Gamma$ and $C^*_r(\Gamma)$ are ($\G$-)co-amenable. 
The groups we work with satisfy a singularity phenomenon described in \cite{BBHP}. 
The setup of singularity allows us to obtain a description of $\Gamma$-invariant intermediate von Neumann subalgebras $L^{\infty}(X,\xi)\subset\mathcal{M}\subset L^{\infty}(X,\xi)\rtimes\Gamma$ in terms of the normal subgroups of $\Gamma$.
\end{abstract}
\tableofcontents

\pagebreak
\section{Introduction}

The notion of \say{singularity} has been used to prove rigidity results for $\Gamma$-operator algebras in various settings, where $\Gamma$ is a discrete countable group. It appears in the works of \cite{kalantar_kennedy_boundaries, Haagerup, HartKal,BC14} etc., where the authors put singular states into use. More recently, \cite{BBHP,bader2021charmenability} used singularity in the context of $\Gamma$-equivariant ucp maps $\Phi:\mathcal{M}\to L^{\infty}(B,\nu)$, where $\mathcal{M}$ is a $\Gamma$-von Neumann algebra and $(B,\nu)$ is a non-singular probability $\Gamma$-space. Such $\Phi$ is called singular if the states on $\mathcal{M}$ given by the dual map for almost every $b\in B$ are singular with respect to their $\Gamma$-translations.

In this paper, we further highlight the role of the singularity of ucp maps for rigidity phenomena. We first make the following definition. 
Let $\mathbb{E}$ denote the canonical conditional expectation on $L^{\infty}(B,\nu)\rtimes\Gamma$. We say that the action has the \say{Singular-Hereditary} property (abbreviated as SH) if for every $\Gamma$-invariant von Neumann algebra $\mathcal{M}\subset L^{\infty}(B,\nu)\rtimes\Gamma $, either $\mathbb{E}|_{\mathcal{M}}$ is $\Gamma$-singular as a ucp map or, $\nu\circ\mathbb{E}|_{\mathcal{M}}$ is a $\Gamma$-invariant state. In our first main result, we use the SH-property, combined with Zimmer amenability, to conclude that all the invariant subalgebras of $C_r^*(\Gamma)$, or of $L(\Gamma)$ are co-amenable.
\begin{theorem}
\thlabel{singularityimpliescoamenability}
Let $\Gamma$ be a countable discrete group. Assume that there exists a non-singular $\Gamma$-space $(B,\nu)$ which has the SH-property and is Zimmer amenable. Then every non-trivial $\Gamma$-$C^*$-subalgebra $\mathcal{A}\subset C_r^*(\Gamma)$ is co-amenable. Similarly, every $\Gamma$-invariant von Neumann subalgebra $\mathcal{M}\subset L(\Gamma)$ is  co-amenable.   \end{theorem}
Here, co-amenability of $\mathcal{A}$ is in the sense of \cite{KalPan}, namely, the commutant $\mathcal{A}'\subset\mathbb{B}(\ell^2(\Gamma))$ admits a $\Gamma$-invariant state (similarly for $\mathcal{M}\subset L(\Gamma)$). Kalantar-Panagopoulos proved the conclusion of \thref{singularityimpliescoamenability} for higher rank lattices using \say{Non-commutative-Nevo-Zimmer} theorem~\cite{BH19}. This neoteric result of Kalantar-Panagopoulos stirred our interest in this problem. 

It is worth pointing out that the non-commutative Nevo-Zimmer theorem \cite[Theorem~B]{BH19} made up one of the key ingredients in the work of \cite{KalPan}. It is known that such a structure theorem cannot hold for semisimple Lie groups admitting a rank one factor. However, there are examples of groups which are a product of rank one factors and yet have an action on a non-singular $\Gamma$ space which has SH-property (see \thref{groupadmmittingSHbutnotNCNZ}).

Let us also note that the conclusion of \thref{singularityimpliescoamenability} can be considered a non-commutative strengthening of Margulis' normal subgroup theorem. Indeed, any subgroup for which the conclusion holds is just-non-amenable (that is, all its normal subgroups are co-amenable). If we assume in addition that the group $\Gamma$ has property (T), then all $\Gamma$-invariant subalgebras are co-finite, and all normal subgroups are of finite index. We remark that the key to the Normal Subgroup Theorem (NST) lies in understanding the structure of the Furstenberg-Poisson boundary.

In addition, it follows from \thref{singularityimpliescoamenability} that if $\Gamma$ is a non-amenable group admitting a Zimmer amenable SH space, then $\Gamma$ has trivial amenable radical. Examples of groups that satisfy the conditions of \thref{singularityimpliescoamenability} can be found in \cite{BBHP,bader2021charmenability}, where NST was shown. In these examples, the source of such $\Gamma$-actions is the Furstenberg-Poisson boundary of a random walk on a locally compact group $(G,\mu)$ associated with $\Gamma$. The setup in these examples contrasts with the structure of higher rank lattices \cite{KalPan}, where the space is a Furstenberg-Poisson boundary of the group $\Gamma$ itself. \\

Since our classification in \thref{singularityimpliescoamenability} is dependent on the structure of the subalgebras of the crossed product $L^{\infty}(B,\nu)\rtimes\Gamma$, we conjecture the following for higher rank lattices.

\begin{conjecture*}
\thlabel{conj}
Let $\Gamma$ be an irreducible lattice in a higher rank semisimple Lie group $G$ with a finite center and no non-trivial compact factor, all of whose simple factors have real rank of at least two. Let $(G/P,\nu_P)$ be the Furstenberg-Poisson boundary associated with a random walk $\mu$ on $\Gamma$. Then, every $\Gamma$-invariant subalgebra of the crossed product $L^{\infty}(G/P,\nu_P)\rtimes\Gamma$ is of the form $L^{\infty}(G/Q,\nu_Q)\rtimes\Lambda$, where $\Lambda\triangleleft\Gamma$ and $Q$ is a Parabolic subgroup of $G$.
\end{conjecture*}
An affirmative answer to the above Conjecture would completely describe the $\Gamma$-invariant subalgebras of the crossed product. We provide a sufficient condition (\thref{upperboundforM}) which implies this conjecture.

Almost all the known results in this direction deal with \say{intermediate algebras} $\mathcal{M}$ of the form $L(\Gamma)\subset\mathcal{M}\subset\mathcal{N}\rtimes\Gamma$ (see e.g., \cite{A19,Suz,Houdayersurvey}). At the same time, there has been considerable work to describe intermediate algebras $\mathcal{M}$ of the form $\mathcal{N}\subset\mathcal{M}\subset \mathcal{N}\rtimes\Gamma$, where $\mathcal{N}$ is a von Neumann algebra (see for example, \cite{choda1978galois,izumi1998galois,Cameron2015BimodulesIC, cameron2016intermediate} to name a few). In this paper, we provide a similar kind of classification for $\Gamma$-invariant subalgebras of $L^{\infty}(X,\xi)\rtimes\Gamma$ containing $L^{\infty}(X,\xi)$, where $(X,\xi)$ is a non-trivial factor of an ergodic non-singular $\Gamma$-space $(B,\nu)$ satisfying the SH-property.
\begin{theorem}
\thlabel{comdessdprop}
Let $(B,\nu)$ be an ergodic non-singular $\Gamma$-space with the SH-property and, let $(X,\xi)$ be a non-trivial factor of $(B,\nu)$. Then, every $\Gamma$-invariant von Neumann algebra $\mathcal{M}$ with $L^{\infty}(X,\xi)\subset\mathcal{M}\subset L^{\infty}(X,\xi)\rtimes\Gamma$ is a crossed product of the form $L^{\infty}(X,\xi)\rtimes\Lambda$ for a normal subgroup $\Lambda\triangleleft\Gamma$.
\end{theorem}
Notice that since $\mathcal{M}$ is not assumed to contain $L(\Gamma)$, it is not automatically $\Gamma$-invariant. Moreover, we cannot say anything about intermediate algebras that are not $\Gamma$-invariant.

\subsection*{Acknowledgement}
We express our gratitude towards Mehrdad Kalantar, Yongle Jiang, and Hanna Oppelmayer for numerous enlightening discussions. We also thank the anonymous referee for taking the time to carefully read our manuscript, and for his/her numerous suggestions and corrections, which have improved the exposition enormously.   
\section{Preliminaries}
Let $\Gamma$ be a discrete countable group and $\mathcal{A}$ be an unital $\Gamma$-$C^*$-algebra. By this, we mean a  $C^*$-algebra $\mathcal{A}$ endowed with the action  $\Gamma\curvearrowright \mathcal{A}$ by $*$-automorphisms such that the map $\Gamma \times \mathcal{A}\xrightarrow{} \mathcal{A}$ which sends $(g, x)\xrightarrow[]{}
g.x$ is continuous. For a von Neumann algebra $\mathcal{M}$ with separable predual $\mathcal{M}_{*}$, we endow $\mathcal{M}$ with the
ultraweak (i.e., weak$^*$) topology coming from the canonical identification $\mathcal{M} =(\mathcal{M}_{*})^*$. Via this identification, $\mathcal{M}_*$ (as a subset of $\mathcal{M}^*$) consists of all ultraweakly continuous linear
functionals, also called normal linear functionals. By a $\Gamma$-von Neumann algebra $\mathcal{M}$, we mean a von Neumann algebra $\mathcal{M}$ equipped with an action  $\Gamma\curvearrowright \mathcal{M}$ by $*$-automorphisms such that the map $\Gamma \times \mathcal{M}\xrightarrow{} \mathcal{M}$ which sends $(g, x)\xrightarrow[]{}
g.x$ is continuous. We briefly recall the notion of boundary structure as defined in \cite{BBHP}. 
\stoptocentries
We denote by $\mathcal{S}\left(\mathcal{A}\right)$, the set of all states on $\mathcal{A}$.
Let us recall the notion of singular states.
The states $\tau,\tilde{\tau}\in\mathcal{S}\left(\mathcal{A}\right)$ are said to be \textbf{singular} ($\tau\perp\tilde{\tau}$) if there exists a net $0\le a_i\le 1\in\mathcal{A}$ such that $\lim_i\tau(a_i)=1$ and $\lim_i\tilde{\tau}(a_i)=0$.

Let $\Phi: \mathcal{M}\to L^{\infty}(B,\nu)$ be a  $\Gamma$-equivariant ucp map. Upon restricting to an ultraweakly dense $\Gamma$-invariant separable $C^*$-subalgebra $\tilde{\mathcal{A}}$, we obtain a $\Gamma$-equivariant map $\theta: B\to \mathcal{S}(\tilde{\mathcal{A}})$. Moreover, for $\nu$-almost every $b\in B$, $\theta(b)\in \mathcal{S}(\tilde{\mathcal{A}})$ is defined by $\Phi(a)(b)=\theta(b)(a)$ for $a\in\tilde{\mathcal{A}}$. We say that $\Phi$ is $\Gamma$-\textbf{singular} if $
s.\theta(b)\perp\theta(b)$ for almost every $b\in B$ and $s\in\Gamma\setminus\{e\}$ (see e.g., \cite[Definition~3.6]{Houdayersurvey}). In particular, that there exists a net $\tilde{a_i}\in \tilde{\mathcal{A}}\cap\mathcal{M}$ with $0\le \tilde{a_i}\le 1$ such that $\lim_i\theta(b)(\tilde{a_i})=1$ and $\lim_ig.\theta(b)(\tilde{a_i})=0$. It follows from \cite[Proposition~4.10]{BBHP} that the notion of $\Gamma$-singularity of $\Phi$ is independent of the choice of $\Tilde{\mathcal{A}}$. Moreover, we say that $\Phi$ is invariant if $\Phi(\mathcal{M})=\mathbb{C}$\subsection*{Weak topologies} We now turn to recall the notions of weak topology and ultraweak-topology on the set of $\mathbb{B}(\mathcal{H})$ of bounded linear maps on $\mathcal{H}$. The readers can refer to \cite{takesaki} for more details on these. The weak operator topology (abbreviated as WOT) is generated by open sets of the form 
\[\left\{T\in\mathbb{B}(\mathcal{H}): \left|\langle(T-T_0)\xi,\eta\rangle\right|<\epsilon\right\},\]
where $T_0\in \mathbb{B}(\mathcal{H}),\xi,\eta\in\mathcal{H}$ and $\epsilon>0$. The ultraweak topology (also known as $\sigma$-weak topology) is the topology induced by the open sets of the form 
\[\left\{T\in\mathbb{B}(\mathcal{H}): \left|\sum_i\langle(T-T_0)\xi_i,\eta_i\rangle\right|<\epsilon\right\},\]
where $T_0\in \mathbb{B}(\mathcal{H}),\xi_i,\eta_i\in\mathcal{H}$ with $\sum_i\|\xi_i\|^2,\sum_i\|\eta_i\|^2<\infty$ and $\epsilon>0$.\\
\noindent
On the closed unit ball of $\mathbb{B}(\mathcal{H})$, the ultraweak topology and the WOT coincide (see \cite[Chapter-II,~Lemma~2.5]{takesaki}).
\subsection*{Ultraweakly dense $C^*$-subalgebra of a von Neumann algebra} Given a von Neumann algebra $\mathcal{M}$ acting on a separable Hilbert space $\mathcal{H}$, we can find an ultraweakly dense $C^*$-subalgebra $\mathcal{A}\subset\mathcal{M}$. We can also choose $\mathcal{A}$ to be separable (in the norm) as well. We shall refer to such a $C^*$-subalgebra as a \say{separable model} of $\mathcal{M}$.  We include a proof of this fact.
\begin{prop}
\thlabel{compactmodel}
Let $\mathcal{M}$ be a von Neumann algebra action on a separable Hilbert space $\mathcal{H}$. Then, there exists a unital $C^*$-algebra $\mathcal{A}\subset\mathcal{M}$ such that $\mathcal{A}$ is ultraweakly dense inside $\mathcal{M}$. Moreover, $\mathcal{A}$ is separable in the norm-topology. In particular, $\mathcal{M}$ has a separable model. 
\begin{proof}
Let $\mathbb{B}_1$ denote the closed unit ball of $\mathbb{B}(\mathcal{H})$. It is compact in WOT, hence in the ultraweak-topology. Since $\mathcal{H}$ is separable, the ultraweak-topology on $\mathbb{B}_1$ is metrizable (see \cite[Chapter-II,~Proposition~2.7]{takesaki}). Therefore, $\mathbb{B}_1$ is separable in the ultraweak topology. Since subsets of separable sets are separable in metric spaces, the unit ball $\mathcal{M}_1$ of $\mathcal{M}$ is separable in the ultraweak topology. Let $A=\{a_n: n\in\mathbb{N}\}$ be a countable dense (in the ultraweak-topology) subset of $\mathcal{M}_1$. By adjoining the unit of $\mathcal{M}$ to $A$ if required, we can assume that $A$ contains the unit of $\mathcal{M}$. Let $\mathcal{A}$ be the unital $C^*$-algebra generated by $A$. We observe that 
\[\left\{\sum_{j=1}^m(c_j+id_j)a_{1j}\ldots a_{nj}: n,m\in\mathbb{N}, c_j,d_j\in\mathbb{Q}\text{ and }a_{1j},\ldots,a_{nj}\in A\cup A^*\right\}\]
is a countable dense subset of $\mathcal{A}$ in the norm-topology. Hence, $\mathcal{A}$ is separable. We now show that $\mathcal{A}$ is ultraweakly dense in $\mathcal{M}$. Let $x\in\mathcal{M}$ and $\epsilon>0$ be given. Consider a basic open set $W^{\epsilon,x}_{\varphi_1,\ldots,\varphi_n}$ around $x$. Note that
\[W^{\epsilon,x}_{\varphi_1,\ldots,\varphi_n}=\left\{y\in\mathcal{M}: \left|\varphi_i(x-y)\right|<\epsilon,~i=1,2,\ldots,n\right\}.\]
Moreover, for each $i=1,2,\ldots,n$, $\varphi_i\in\mathbb{B}(\mathcal{H})_*$ is given by 
\[\varphi_i(\cdot)=\sum_{j}\left\langle(\cdot)\xi^i_j,\eta^i_j\right\rangle,~\xi^i_j,\eta^i_j\in\mathcal{H},~\sum_j\|\xi^i_j\|^2,~\sum_j\|\eta^i_j\|^2<\infty.\]
Let $m\in\mathbb{N}$ be such that $\|x\|<m$. Then, $\frac{1}{m}x\in\mathcal{M}_1$. Since $A$ is ultraweakly dense inside $\mathcal{M}_1$, there exists $n_0\in\mathbb{N}$ such that $$a_{n_0}\in W^{\frac{\epsilon}{m},\frac{x}{m}}_{\varphi_1,\ldots,\varphi_n}=\left\{y\in\mathcal{M}: \left|\varphi_i\left(\frac{1}{m}x-y\right)\right|<\frac{\epsilon}{m},~i=1,2,\ldots,n\right\}.$$This in particular implies that for each $i=1,2,\ldots,n$,
\[\left|\varphi_i\left(x-ma_{n_0}\right)\right|=m\left|\varphi_i\left(\frac{1}{m}x-a_{n_0}\right)\right|<m\frac{\epsilon}{m}=\epsilon.\]
Therefore, $ma_{n_0}\in W^{\epsilon,x}_{\varphi_1,\ldots,\varphi_n}$. Since $ma_{n_0}\in\mathcal{A}$, it follows that $\mathcal{A}$ is ultraweakly dense in $\mathcal{M}$.
\end{proof}
\end{prop}

\subsection*{Crossed product von Neumann algebra}
We briefly recall the construction of the crossed product von Neumann algebra. Let $\mathcal{M}$ be a $\Gamma$-von Neumann algebra. 
Given a Hilbert space $\mathcal{H}$, let $\ell^2(\Gamma,\mathcal{H})$ be the space of square summable $\mathcal{H}$-valued functions on $\Gamma$, i.e.,
\[\ell^2(\Gamma,\mathcal{H})=\left\{\xi:\Gamma\to \mathcal{H}\text{ such that }\sum_{h\in\Gamma}\|\xi(h)\|_{\mathcal{H}}^2<\infty.\right\}\]
There is a natural action $\Gamma\curvearrowright \ell^2(\Gamma,\mathcal{H})$ by left translation:
\[\lambda_g\xi(h):=\xi(g^{-1}h), \xi \in \ell^2(\Gamma,\mathcal{H}), g,h \in \Gamma\]
Given a faithful $*$-representation $\pi : \mathcal{M}\to \mathbb{B}(\mathcal{H})$ of a $\Gamma$-von Neumann algebra $\mathcal{M}$ into the space
of bounded operators on the Hilbert space $\mathcal{H}$, let $\sigma$ be the $*$-representation \[\sigma:\mathcal{M} \to B(\ell^2(\Gamma,\mathcal{H}))\] defined by \[\sigma(a)(\xi)(h):=\pi(h^{-1}a)\xi(h), a \in \mathcal{M}\]
where $\xi \in \ell^2(\Gamma,\mathcal{H})$, $h \in \Gamma$. The von Neumann crossed product
$\mathcal{M}\rtimes\Gamma$ is generated (as a von Neumann algebra inside $\mathbb{B}(\ell^
2
(\Gamma, \mathcal{H})$), by the left regular representation $\lambda$ of $\Gamma$ and the faithful $*$-representation $\sigma$ of $\mathcal{M}$ in $\mathbb{B}(\ell^2
(\Gamma, \mathcal{H}))$.
Moreover, this representation translates the action $\Gamma \curvearrowright\mathcal{M}$ into an inner action
by the unitaries $\{\lambda(g), g \in \Gamma\}$. It follows from the construction that $\mathcal{M}\rtimes\Gamma$ contains $L(\Gamma)$ as as a von Neumann-subalgebra. The von Neumann crossed product $\mathcal{M}\rtimes\Gamma$ comes equipped with a $\Gamma$-equivariant faithful normal conditional expectation $\mathbb{E}:\mathcal{M}\rtimes\Gamma\to\mathcal{M}$ defined by 
\[\mathbb{E}\left(\sigma(a_g)\lambda_g\right)=\left\{ \begin{array}{ll}
0 & \mbox{if $g\ne e$}\\
\sigma(a_e) & \mbox{otherwise}\end{array}\right\}\]
We are now ready to define an SH-space.
\begin{definition}[Singular Hereditary Space]
Let $(B,\nu)$ be an ergodic non-singular $\Gamma$-space. We say that the action $\Gamma\curvearrowright (B,\nu)$ has \say{singular hereditary property} if for every $\Gamma$-invariant von Neumann algebra $\mathcal{M}\subset L^{\infty}(B,\nu)\rtimes\Gamma$, either $\mathbb{E}|_{\mathcal{M}}$ is $\Gamma$-singular or $\mathbb{E}\left(\mathcal{M}\right)=\mathbb{C}$. In this case, we say that $(B,\nu)$ is an SH-space. 
\end{definition}   

One can view the definition of SH-spaces as a non commutative analog of the case where the action on $(B,\nu)$, and on all of its non trivial factors, is essentially free. Examples of SH-spaces originate from the works of \cite{BBHP} and \cite{bader2021charmenability}.
\begin{example}
\thlabel{genconditionforsh}
Let $\Gamma$ be a discrete group having trivial amenable radical which satisfies the condition~$(a)$ in \cite[Proposition~4.17]{BBHP}. We point out that concrete examples of such groups have been provided in \thref{groupadmmittingSHbutnotNCNZ}. We now claim that the space $(B,\nu)$ mentioned there is an SH-space. Indeed, let $\mathcal{M}$ be a $\Gamma$-invariant subalgebra of $L^{\infty}(B,\nu)\rtimes\Gamma$ and $\mathbb{E}$ be the canonical conditional expectation associated with the crossed product. Then, letting $M=\M$ and $E=\mathbb{E}$, it follows from condition~$(a)$ that either $E|_{M}$ is either $\Gamma$-singular or invariant. Suppose that $E|_{M}$ is not $\Gamma$-singular.  Then, $E|_{M}$ being invariant in the sense of \cite[Definition~4.1]{BBHP} means that $E(M)\subset L^{\infty}(B,\nu)^{\Gamma}$. Since $(B,\nu)$ is an ergodic space (even metrically ergodic), it follows that $E(M)=\mathbb{C}$. 
\end{example}

We now discuss an example of a group for which the non-commutative Nevo-Zimmer theorem does not hold and admits an SH-space. 
\begin{example}\cite[Theorem~D]{BBHP}
\thlabel{groupadmmittingSHbutnotNCNZ}
Let $T$ be a bi-regular tree. We denote by $\text{Aut}^{+}(T)$, the group of bi-coloring preserving automorphisms of $T$ which acts $2$-transitively on the boundary $\partial T$.
Assume that $n\ge 2$. For each $i=1,2,\ldots,n$, let $G_i$ be a closed subgroup of the bi-regular tree $\text{Aut}^{+}(T_i)$. Moreover, let $\Gamma$ be
a co-compact lattice in $G=G_1\times\ldots\times G_n$ with dense projections. Note that the non-commutative Nevo-Zimmer theorem does not hold for $\Gamma$. Now, for each $i=1,2,\ldots,n$, let $B_i=\partial T_i$. Moreover, equipped with the right measure $\nu_i$, $(B_i,\nu_i)$ is the Furstenberg-Poisson boundary of $G_i$ for some generating measure $\mu_i$ on $G_i$ (see the discussion in the proof of \cite[Theorem~D]{BBHP} and \cite[Theorem~5.1]{BadShal06}). It follows from \cite[Corollary~3.2]{BadShal06} that $(B,\nu)=\left(\prod_{i=1}^nB_i,\otimes_{i=1}^n\nu_i\right)$ is the Furstenberg-Poisson boundary of $G$. Arguing similarly as in the proof of \cite[Theorem~D]{BBHP}, we obtain that the action $\Gamma\curvearrowright (B,\nu)$ is ergodic and Zimmer-amenable. It follows from the $2$-transitivity assumption that the group $\Gamma$ has a trivial amenable radical. Now, it is shown in \cite[Theorem~D]{BBHP} that the group $\Gamma$ satisfies the condition~$(a)$ in \cite[Proposition~4.17]{BBHP}. As a consequence, it follows from \thref{genconditionforsh} that $(B,\nu)$ is an SH-space.
\end{example}
We also provide an example of a group to which the non-commutative Nevo-Zimmer theorem applies and, as an upshot, accedes an SH-space.  
\begin{example}\thlabel{example2}\cite{bader2021charmenability}
\noindent
Let $k$ be any local field. Let $\textbf{G}$ be any almost $k$-simple connected algebraic group with real rank $\text{rank}_k(\textbf{G})\ge 2$. Let $\textbf{P}<\textbf{G}$ be a minimal parabolic $k$-subgroup. Set $G=\textbf{G}(k)$ and $P=\textbf{P}(k)$. Let $\Gamma < G$ be a lattice, equipped with a Furstenberg measure, i.e., a measure for which there exists a measure $\nu_P$ on $G/P$ such that $(G/P,\nu_P)$ is a Furstenberg-Poisson boundary.  We shall argue that $(G/P,\nu_P)$ is an SH-space.  Let us begin by observing that the action $\Gamma\curvearrowright (G/P,\nu_P)$ is essentially free and ergodic (\cite[Lemma~6.2]{BBHP}). Let $\mathcal{M}\subset L^{\infty}(G/P,\nu_P)\rtimes\Gamma$ be an invariant subalgebra. Arguing similarly as in \cite[Lemma~2.16]{KalPan}, we see that the action $\Gamma\curvearrowright\M$  is ergodic, i.e., $\mathcal{M}^{\Gamma}=\mathbb{C}$. We can now appeal to \cite[Theorem~5.4]{bader2021charmenability} to conclude that either $\mathbb{E}(\M)=\mathbb{C}$ or $\mathbb{E}|_{\M}$ is $\Gamma$-singular.
\end{example}
Let us also note that we shall identify $\mathbb{B}(\ell^2(\Gamma))$ as a $\Gamma$-invariant subalgebra of $\mathbb{B}(\ell^2(\Gamma,\mathcal{H}))$. Under this identification, it immediately follows that for any $\Gamma$-invariant subalgebra $\mathcal{A}\subset \mathbb{B}(\ell^2(\Gamma))$, the relative commutant $\mathcal{A}'\cap \mathbb{B}(\ell^2(\Gamma))$ is contained inside $\tilde{\mathcal{A}}'$, the commutant of $\mathcal{A}$ inside $\mathbb{B}(\ell^2(\Gamma),\mathcal{H})$.

We end this section with the following easy observation which allows us to relate the commutant of $\mathcal{A}$ (or, $\M$) inside $\mathbb{B}\left(\ell^2(\Gamma, \cH\right)$ for $\cH=L^2(B,\nu)$ to that of the relative commutant inside $L^{\infty}(B,\nu)\rtimes\Gamma$.
\begin{lemma}
\thlabel{relatingthecommutants}
Let $(B,\nu)$ be a non-singular $\Gamma$-space. Suppose that $\mathcal{A}$ (or, $\M$) is a $\Gamma$-invariant $C^*$-subalgebra (or, von Neumann subalgebra) of $L^{\infty}(B,\nu)\rtimes\Gamma$ such that there exists a ucp map $\Phi:\mathbb{B}(\ell^2(\Gamma,L^2(B,\nu))\to L^{\infty}(B,\nu)\rtimes\Gamma$ with $\Phi|_{L^{\infty}(B,\nu)\rtimes\Gamma}=\text{id}$. 

Then, $\Phi$ maps $\mathcal{A}'\cap\mathbb{B}(\ell^2(\Gamma,L^2(B,\nu))$ (similarly, $\M'$) to the respective relative commutants inside $L^{\infty}(B,\nu)\rtimes\Gamma$. Moreover, the map $\Phi|_{\mathcal{A}'}$ or $\Phi|_{\mathcal{M}'}$ is surjective.
\begin{proof}
Let $\M$ be a $\Gamma$-invariant von Neumann subalgebra of $L^{\infty}(B,\nu)\rtimes\Gamma$.
Let $T\in\mathbb{B}(\ell^2(\Gamma,L^2(B,\nu))$ be such that $Tx=xT$ for all $x \in \mathcal{M}$. Then, applying $\Phi$ on both sides, we obtain that $\Phi(Tx)=\Phi(xT)$ for all $x\in \M$. Since $\Phi|_{L^{\infty}(B,\nu)\rtimes\Gamma}=\text{id}$, $L^{\infty}(B,\nu)\rtimes\Gamma$ falls in the multiplicative domain of $\Phi$ (see \cite[Proposition~1.5.7]{BroOza08}). Therefore, for all $x\in\mathcal{M}$,
we obtain that 
\begin{align*}\Phi(T)x=\Phi(T)\Phi(x)=\Phi(Tx)=\Phi(xT)=\Phi(x)\Phi(T)=x\Phi(T)\end{align*}
Consequently, it follows that $\Phi(T)\in \M'\cap (L^{\infty}(B,\nu)\rtimes\Gamma)$. The proof for a $\Gamma$-invariant $C^*$-subalgebra follows vis a vis to the above argument. The surjectivity of the map $\Phi|_{\mathcal{A}'}$(similarly, for $\Phi|_{\mathcal{M}'}$) follows from the fact that $\Phi|_{L^{\infty}(B,\nu)\rtimes\Gamma}=\text{id}$.
\end{proof}
\end{lemma}

\starttocentries
\section{The Singular Hereditary Property}
The key ingredient in the proof of \cite{KalPan} is the deep structural non-commutative-Nevo-Zimmer Theorem (see \cite[Theorem~B]{BH19}). However, such a phenomenon is only observed in the case of higher-rank lattices. To prove co-amenability in our setup, we needed to use instead, the singular hereditary property. 
The following proposition establishes the link between an invariant algebra and its relative commutant in the crossed product if we know that the second object is singular (also see \cite[Lemma~2.2]{HartKal} and \cite[Proposition~3.7]{Houdayersurvey}).
\begin{prop}
\thlabel{singularlift}
Let $(X,\nu)$ be a non-singular $\Gamma$-space.
Let $\mathcal{M}\subset L^{\infty}(X,\nu)\rtimes\Gamma$ be a $\Gamma$-invariant subalgebra.
Suppose that the relative commutant $\tilde{\mathcal{M}}$ of $\mathcal{M}$ in $L^{\infty}(X,\nu)\rtimes\Gamma$ is $\Gamma$-singular (that is,  $\mathbb{E}|_{\tilde{\mathcal{M}}}$ is $\Gamma$-singular). Then, $\mathbb{E}\left(a\lambda(g)\right)=0$ for all $a \in\mathcal{M}$ and for all $g \in \Gamma\setminus\{e\}$. 
\end{prop}

\begin{proof}
We shall complete the proof in three steps.

\underline{\textit{Step-1}}: Choose a separable model $\tilde{\mathcal{A}}\subset L^{\infty}(X,\nu)\rtimes\Gamma$ such that $\Tilde{\mathcal{A}}$ contains $\lambda(\Gamma)$, a separable model $\mathcal{A}$ of $\M$, and a separable model $\mathcal{A}_1$ of $\Tilde{\mathcal{M}}$.   

Let $\tilde{\M}_1$ denote the unit ball of $\tilde{\M}$. It follows from the first part of the proof of \thref{compactmodel} that $\tilde{\M}_1$ is separable in the ultraweak-topology. Let $\tilde{A}_1$ be a countable dense (in the ultraweak-topology) subset of $\tilde{\M}_1$. By adjoining $\tilde{A}_1$ with $\{\lambda(s)a\lambda(s)^*: a\in\Tilde{A}_1,s\in\Gamma\}$, we shall assume that $\tilde{A}_1$ is $\Gamma$-invariant. Likewise, we can find a $\Gamma$-invariant countably dense (in the ultraweak-topology) subset $M_1$ of the unit ball of $\mathcal{M}$. Similarly, we can also find a countably dense (in the ultraweak-topology) subset $A_1$ of the unit ball of $L^{\infty}(X,\nu)\rtimes\Gamma$. Let $\tilde{\mathcal{A}}$ be the $C^*$-algebra generated by $\tilde{A}_1$, $M_1$, $A_1$ and $\lambda(\Gamma)$, i.e., 
\[\tilde{\mathcal{A}}=C^*\left(\tilde{A}_1\cup M_1\cup A_1\cup\lambda(\Gamma)\right).\]
Moreover, let $\mathcal{A}_1$ be the $C^*$-algebra generated by $\tilde{A}_1$ and $\lambda(e)$, and $\mathcal{A}$, the $C^*$-algebra generated by $M_1$ and $\lambda(e)$. It follows from the later part of the proof of \thref{compactmodel} that $\Tilde{\mathcal{A}}$, $\mathcal{A}$ and $\mathcal{A}_1$ are separable models for $L^{\infty}(X,\nu)\rtimes\Gamma$, $\M$ and $\Tilde{\mathcal{M}}$ respectively. Moreover, it is evident from the construction that $\Tilde{\mathcal{A}}$ contains $\mathcal{A}$, $\mathcal{A}_1$ and $\lambda(\Gamma)$.

\underline{\textit{Step-2}}: $\mathbb{E}(a\lambda(g))=0$ for all $a\in\mathcal{A}$ and $g\in\Gamma\setminus\{e\}$.\\
Note that $\Gamma\curvearrowright\tilde{\mathcal{A}}$ by conjugation. 
Restrict $\mathbb{E}$ to $\tilde{A}$ and denote by  $\theta:X\to\mathcal{S}(\tilde{\mathcal{A}})$ the corresponding $\Gamma$-equivariant measurable map. Since $\mathcal{A}_1$ is a separable model for $\tilde{\mathcal{M}}$, using the uniqueness of the map, we see that $\tilde{\theta}:X\to \mathcal{S}(\mathcal{A}_1)$ is given by $\tilde{\theta}(b)(a)=\theta(b)(a)$ for $a\in\mathcal{A}_1$. Let $g\in\Gamma\setminus\{e\}$. Since $\tilde{\mathcal{M}}$ is $\Gamma$-singular, we can find $\tilde{X}\subset X$ a co-null measure subset such that for every $x\in\tilde{X}$,
$\tilde{\theta}(x)\perp g.\tilde{\theta}(x)$. Fix $x\in \tilde{X}$. It follows that there exists a net $\tilde{a_i}\in \mathcal{A}_1$ with $0\le \tilde{a_i}\le 1$ such that $\lim_i\tilde{\theta}(x)(\tilde{a_i})=1$ and $\lim_ig.\tilde{\theta}(x)(\tilde{a_i})=0$. This in particular shows that $\theta(x)\perp g.\theta(x)$.

We first note that $\theta(x)$ is a state on $\Tilde{\mathcal{A}}$. Since $\mathcal{A}$ and $\lambda(\Gamma)$ are both contained in $\Tilde{\mathcal{A}}$, $\theta(x)(a\lambda(g))$ makes sense for all $a\in\mathcal{A}$ and for all $g\in\Gamma\setminus\{e\}$.
Let $\tau=\theta(x)$. Now, since $\tilde{a_i}a=a\tilde{a_i}$, we see that
\begin{align*}
\left|\tau(\tilde{a_i}a\lambda(g))\right|^2&=\left|\tau\left(a\tilde{a_i}^{\frac{1}{2}}\tilde{a_i}^{\frac{1}{2}}\lambda(g)\right)\right|^2\\&\le\tau\left(a\tilde{a_i}a^*\right)\tau\left(\lambda(g^{-1})\tilde{a_i}\lambda(g)\right)\\&=\tau(a\tilde{a_i}a^*)g.\tau(\tilde{a}_i)
\end{align*}
Therefore, we obtain that \[\lim_i\tau(\tilde{a_i}a\lambda(g))=0\]
On the other hand,
\begin{align*}
&\lim_i\tau\left((1-\tilde{a}_i)a\lambda(g)\right)\\&=\lim_i\tau\left((1-\tilde{a}_i)^{\frac{1}{2}}(1-\tilde{a}_i)^{\frac{1}{2}}a\lambda(g)\right)\\& \le\lim_i\left\|\tau\left((1-\tilde{a}_i)\right)\right\|^{\frac{1}{2}}\|\tau\left((\lambda(g^{-1})a^*(1-\tilde{a}_i)a\lambda(g)\right)\|^{\frac{1}{2}}\\&=0.
\end{align*}
Now, combining the above two identities, we see that
\[\tau(a\lambda(g))=\lim_i\tau(\tilde{a}_ia\lambda(g))+\lim_i\tau\left((1-\tilde{a}_i)a\lambda(g)\right)=0.\]
In particular, we obtain that $\theta(x)(a\lambda(g))=0$ for all $x \in \tilde{X}$. This in turn implies that $\mathbb{E}(a\lambda(g))=0$ for all $a\in \mathcal{A}$ and $g\in\Gamma\setminus\{e\}$.

\underline{\textit{Step-3}}: $\mathbb{E}(a\lambda(g))=0$ for all $a\in\mathcal{M}$ and for all $g\ne e$.\\ 
Let $a\in\mathcal{M}$ and $g\in\Gamma\setminus\{e\}$ be given. Since $\mathcal{A}$ is ultraweakly dense in $\mathcal{M}$, we can find a net $a_i\in\mathcal{A}$ such that $a_i\xrightarrow[]{\text{ultraweakly}}a$. Therefore, $a_i\lambda(g)\xrightarrow[]{\text{ultraweakly}}a\lambda(g)$. Since $\mathbb{E}$ is normal, it follows that $\mathbb{E}(a\lambda(g))=\lim_i\mathbb{E}(a_i\lambda(g))=0$. \end{proof}

Let us now discuss our proof strategy for \thref{singularityimpliescoamenability}. Let $\mathcal{A}\subset C_r^*(\Gamma)$ (or $\M\subset L(\Gamma)$) be a $\Gamma$-invariant subalgebra. We shall use \thref{singularlift} combined with the {SH}-property of the action $\Gamma\curvearrowright (B,\nu)$ to conclude that the relative commutant of the subalgebra in the crossed product $L^{\infty}(B,\nu)\rtimes\Gamma$ has a $\Gamma$-invariant state. From this point onwards, our method varies from our predecessor's in \cite{KalPan}. 

Since $(B,\nu)$ is not the Poisson boundary associated with a random walk on $\Gamma$, we can no longer use Izumi's isomorphism theorem \cite[Theorem~4.1]{izumi2004non}. Instead, we shall use the Zimmer amenability of the action $\Gamma\curvearrowright (B,\nu)$ to conclude that the commutant of $\mathcal{A}$ (or $\mathcal{M}$) inside $\mathbb{B}\left(\ell^2(\Gamma,L^2(B,\nu))\right)$ has a $\Gamma$-invariant state.

\begin{proof}[Proof of \thref{singularityimpliescoamenability}]
Let $(B,\nu)$ be a Zimmer amenable SH-space. We first prove the result in the setting of $C_r^*(\Gamma)$.
Assume that $\mathcal{A}\subset C_r^*(\Gamma)$ is a $\Gamma$-invariant non-trivial subalgebra. We shall show that the commutant $\tilde{\mathcal{M}}=\mathcal{A}'$ contained in $\mathbb{B}(\ell^2(\Gamma,L^2(B,\nu)))$) has a $\Gamma$-invariant state. 

Denote by $\mathcal{M}_1$, relative commutant of $\mathcal{A}$ inside $L^{\infty}(B,\nu)\rtimes\Gamma$. Since $(B,\nu)$ is an SH-space, it follows that either $\mathbb{E}|_{\mathcal{M}_1}$ is $\Gamma$-singular or $\mathbb{E}\left(\mathcal{M}_1\right)=\mathbb{C}$. Let us now argue that the former cannot happen, i.e., we shall show that $\tau|_{\mathcal{M}_1}$ is $\Gamma$-invariant, where $\tau=\nu\circ\mathbb{E}$. We would like to point out that whenever we write $\nu\circ\mathbb{E}$, we think of $\nu$ as a state on $L^{\infty}(B,\nu)$ given by the integration with respect to $\nu$. 
For the sake of contradiction, let us assume that the relative commutant $\mathcal{M}_1$ is $\Gamma$-singular. 
We denote by $\tau_0$ the canonical trace on $C_r^*(\Gamma)$. 
Fix $g\ne h\in \Gamma\setminus\{e\}$ and $\tilde{a}
\in\mathcal{A}$. Since $\mathbb{E}|_{C_r^*(\Gamma)}=\tau_0$, using \thref{singularlift}, we see that 
\begin{align*}
\left\langle\tilde{a}\delta_g,\delta_h\right\rangle&=\left\langle\tilde{a}\lambda(g)\delta_e,\lambda(h)\delta_e\right\rangle\\&=\left\langle\lambda(h^{-1})\tilde{a}\lambda(g)\delta_e,\delta_e\right\rangle\\&=\tau_0(\lambda(h^{-1})\tilde{a}\lambda(g))\\&=\tau_0(\tilde{a}\lambda(gh^{-1}))\\&=0
\end{align*}
Let $\mathcal{E}:\mathbb{B}(\ell^2(\Gamma))\to \ell^{\infty}(\Gamma)$ be the projection onto the diagonal part, i.e.,
\[\mathcal{E}(T)(\delta_g)=\langle T(\delta_g),\delta_g\rangle\delta_g,~T\in\mathbb{B}(\ell^2(\Gamma)),g\in\Gamma \]
Considering $\tilde{a}$ as an element in $\mathbb{B}(\ell^2(\Gamma))$, we can write 
\begin{align*}\tilde{a}(\delta_g)&=\sum_{h\in\Gamma}\langle\tilde{a}(\delta_g),\delta_h\rangle\delta_h\\&=\langle\tilde{a}(\delta_g),\delta_g\rangle\delta_g\\&=\mathcal{E}(\tilde{a})(\delta_g)\end{align*}
Therefore, it follows that $\tilde{a}\in \ell^{\infty}(\Gamma)\cap C_r^*(\Gamma)=\mathbb{C}$. Since $\tilde{a}$ is an arbitrary element, it follows that $\mathcal{A}=\mathbb{C}$.
This, in turn, leads to a contradiction since we assumed $\mathcal{A}$ to be non-trivial in the beginning. As a result, it follows that $\tau|_{\mathcal{M}_1}$ is invariant. 
A similar argument applies to a non-trivial $\Gamma$-invariant von Neumann subalgebra $\mathcal{M}\subset L(\Gamma)$. Alternatively, we can also argue the following. Let $\tilde{\M}_1$ denote the relative commutant of $\mathcal{M}$ inside $L^{\infty}(B,\nu)\rtimes\Gamma$. Suppose that $\mathbb{E}|_{\tilde{\M}_1}$ is $\Gamma$-singular. Let $\tilde{a}\in\mathcal{M}$. We can now appeal to \thref{singularlift} to conclude that $\mathbb{E}(a\lambda(g^{-1}))=0$ for all non-identity elements $g \in \Gamma$. Since the family $\left\{\mathbb{E}(\tilde{a}\lambda(g^{-1})):~g\in\Gamma\right\}$ completely determines $\tilde{a}$ (see for example, the discussion following Lemma~7.5 in \cite{takesaki}), it follows that $\tilde{a}\in\mathbb{C}$. Since $\tilde{a}\in\mathcal{M}$ is arbitrary, this implies that $\mathcal{M}=\mathbb{C}$ which contradicts the non-triviality of $\mathcal{M}$. Hence, we obtain that $\tau|_{\tilde{\M}_1}$ is invariant.

Now, since $\Gamma\curvearrowright (B,\nu)$ is Zimmer amenable, we obtain a projection  $\Phi:\mathbb{B}(\ell^2(\Gamma,L^2(B,\nu)))\xrightarrow[]{}L^{\infty}(B,\nu)\rtimes\Gamma$ (cf. \cite[Theorem~2.1]{zimmer1977hyperfinite}). Since $\Phi|_{L(\Gamma)}=\text{id}$, using \thref{relatingthecommutants}, we obtain that $\Phi$ maps $\mathcal{A}'$ (similarly, $\M'$) to the respective relative commutants inside $L^{\infty}(B,\nu)\rtimes\Gamma$. Consequently, the composition of the restriction of $\tau|_{\M_{1}}$(or, $\tau|_{\tilde{\M}_{1}}$) with $\Phi|_{\mathcal{A}'}$(or, $\Phi|_{\mathcal{M}'}$) gives us an invariant state on $\mathcal{A}'$ (or, $\mathcal{M}'$ respectively). 
\end{proof}
\begin{remark} We can also deduce the co-amenability of the $C^*$-algebra case by arguing similarly as in the proof of \cite[Corollary~5.7]{chifan2022invariant}. We include the proof which was kindly provided to us by the anonymous reviewer. For a $\Gamma$-invariant $C^*$-algebra $\mathcal{A}\subset C_r^*(\Gamma)$, consider $\mathcal{M}=\mathcal{A}''\cap L(\Gamma)$, the von Neumann algebra generated by $\mathcal{A}$ inside $L(\Gamma)$. Now, it follows from the von Neumann algebra case above that there is a $\Gamma$-invariant state on $\mathcal{M}'\cap\mathbb{B}(\ell^2(\Gamma))$. Since $\mathcal{M}'\cap\mathbb{B}(\ell^2(\Gamma))=\mathcal{A}'\cap\mathbb{B}(\ell^2(\Gamma)$, the claim follows.
\end{remark}
\section{Correspondence of invariant algebras for SH-actions}
In this section, we give a description of the $\Gamma$-invariant intermediate algebras $\mathcal{M}$ associated with $L^{\infty}(X,\nu)\subset\mathcal{M}\subset L^{\infty}(X,\nu)\rtimes\Gamma$ for essentially free $\Gamma$-space $(X,\nu)$ with the singular hereditary property. We begin with the following definition.

We would like to point out that in \cite{cameron2016intermediate},  a correspondence was obtained for intermediate von Neumann algebras $\mathcal{N}$ of the form $\mathcal{M}\subset\mathcal{N}\subset\mathcal{M}\rtimes\Gamma$ for a $\Gamma$-von Neumann algebras $\mathcal{M}$ on which the action $\Gamma\curvearrowright\mathcal{M}$ is by properly outer $*$-automorphisms (also see \cite[Corollary~4.5]{Cameron2015BimodulesIC} and the remark thereafter). 

We begin with the following observation, which is essentially contained in \cite[Theorem~3.7]{chifan2020rigidity}.

\begin{lemma}
\thlabel{relationtocommutant}
Let $\tilde{\mathcal{M}}\subset\mathcal{M}$ be an inclusion of von Neumann algebras with expectation, and let $u$ be a unitary element in $\mathcal{M}$ such that $\tilde{\mathcal{M}}$ is invariant under the conjugation by $u$. Let $\mathbb{E}_{\tilde{\M}}:\mathcal{M}\to\tilde{\mathcal{M}}$ be a conditional expectation, then $\mathbb{E}_{\M}(u)u^*\in \tilde{\mathcal{M}}'\cap{\mathcal{M}}$.
\end{lemma}
\begin{proof}
For $x\in \tilde{\mathcal{M}}$, we need to show that $x\mathbb{E}_{\tilde{\M}}(u)u^*=\mathbb{E}_{\tilde{\M}}(u)u^*x$. Indeed, let us observe that
\[\mathbb{E}_{\tilde{\M}}(u)u^*x=\mathbb{E}_{\tilde{\M}}(u)u^*xuu^*=\mathbb{E}_{\tilde{\M}}(uu^*xu)u^*=\mathbb{E}_{\tilde{\M}}(xu)u^*=x\mathbb{E}_{\tilde{\M}}(u)u^*\]
\end{proof}
In general, given an inclusion of unital von Neumann algebras $\tilde{\mathcal{M}}\subset\mathcal{M}$, there may not be a conditional expectation from $\mathcal{M}$ onto $\tilde{\mathcal{M}}$. However, if the inclusion  $\tilde{\mathcal{M}}\subset\mathcal{M}$ is Cartan, then every intermediate von Neumann algebra $\hat{\mathcal{M}}$ of the form $\tilde{\mathcal{M}}\subset\hat{\mathcal{M}}\subset\mathcal{M}$ is in the image of a normal (even faithful) conditional expectation \cite{Yamanouchi2019intermediate}.

In our context, we only need to deal with intermediate von Neumann algebra $\mathcal{N}$ of the form $L^{\infty}(X,\nu)\subset\mathcal{N}\subset L^{\infty}(X,\nu)\rtimes\Gamma$ for non-singular essentially free $\Gamma$-spaces $(X,\nu)$. It is well known that the inclusion $L^{\infty}(X,\nu)\subset L^{\infty}(X,\nu)\rtimes\Gamma$ is Cartan if the action $\Gamma\curvearrowright (X,\nu)$ is essentially free. And hence, every intermediate von Neumann algebra $L^{\infty}(X,\nu)\subset\mathcal{N}\subset L^{\infty}(X,\nu)\rtimes\Gamma$ lies in the image of a faithful normal conditional expectation. 

In fact, for $\Gamma$-von Neumann algebras $\mathcal{M}$, where the action $\Gamma\curvearrowright\mathcal{M}$ is by properly outer $*$-automorphisms, every intermediate von Neumann algebra $\mathcal{M}\subset\mathcal{N}\subset \mathcal{M}\rtimes\Gamma$ lies in the image of a faithful normal conditional expectation~\cite[Theorem~3.2]{cameron2016intermediate}. The notion of properly outer $*$-automorphisms coincides with that of essential freeness for commutative von Neumann algebras.\\
 
\noindent
Now, let $(B,\nu)$ be an SH-space. Let us further assume $L^{\infty}(X,\xi)$ to be a $\Gamma$-invariant subalgebra of $L^{\infty}(B,\nu)$ with the property that the action $\Gamma\curvearrowright (B,\nu)$ restricted to $(X,\xi)$ is essentially free. Let $\mathcal{M}$ be an intermediate von Neumann algebra of the form $L^{\infty}(X,\xi)\subset\mathcal{M}\subset L^{\infty}(X,\xi)\rtimes\Gamma$ lying in the image of a faithful normal conditional expectation $\mathbb{E}_{\mathcal{M}}$. Then, $$\tau(x):=\nu|_{L^{\infty}(X,\xi)}\circ\mathbb{E}\circ\mathbb{E}_{\M}(x),~x\in L^{\infty}(X,\xi)\rtimes\Gamma$$ is a faithful normal state on $L^{\infty}(X,\xi)\rtimes\Gamma$. We can then define the $\|.\|_2$-norm on $L^{\infty}(X,\xi)\rtimes\Gamma$ associated with $\tau$, defined by
\[\|x\|_2=\sqrt{\tau(x^*x)},~x\in L^{\infty}(X,\xi)\rtimes\Gamma\]
The $\|.\|_2$-norm is continuous with respect to the $\sigma$-strong topology, and induces the $\sigma$-strong topology on any bounded (in the operator norm) subset of $L^{\infty}(X,\xi)\rtimes\Gamma$. \stoptocentries
\begin{remark}
\label{norminequality}
In the above setup, the $\|.\|_2$-norm is continuous with respect to $\mathbb{E}_{\M}$, i.e., $\|\mathbb{E}_{\M}(x-y)\|_2\le\|x-y\|_2,~x,y\in L^{\infty}(X,\xi)\rtimes\Gamma$.
Indeed, for $x,y\in L^{\infty}(X,\xi)\rtimes\Gamma$, which follows easily using the Kadison-Cauchy-Schwartz inequality for the ucp map $\mathbb{E}$.\\
\end{remark}

\noindent
We now proceed to give a complete description of intermediate von Neumann algebras $\mathcal{M}$ of the form $L^{\infty}(X,\xi)\subset\mathcal{M}\subset L^{\infty}(X,\xi)\rtimes\Gamma$.

\begin{prop}
\thlabel{comdessdprop}
Let $(B,\nu)$ be an SH-space, $L^{\infty}(X,\xi)\subset L^{\infty}(B,\nu)$ a $\Gamma$-invariant subalgebra with the property that the action $\Gamma\curvearrowright (B,\nu)$ restricted to $(X,\xi)$ is essentially free. Then, every intermediate $\Gamma$-invariant von Neumann algebras $\mathcal{M}$ of the form $L^{\infty}(X,\xi)\subset\mathcal{M}\subset L^{\infty}(X,\xi)\rtimes\Gamma$ is a crossed product of the form $L^{\infty}(X,\xi)\rtimes\Lambda$ for a normal subgroup $\Lambda\triangleleft\Gamma$.
\begin{proof}
Let $\mathcal{M}$ be an intermediate $\Gamma$-invariant von Neumann algebra of the form $L^{\infty}(X,\xi)\subset\mathcal{M}\subset L^{\infty}(X,\xi)\rtimes\Gamma$. Since the action $\Gamma\curvearrowright \left(X,\xi\right)$ is non-singular and essentially free, we can use \cite[Section~4]{Yamanouchi2019intermediate} or \cite[Theorem~3.2]{cameron2016intermediate} to conclude the existence of a faithful normal conditional expectation $\mathbb{E}_{\mathcal{M}}:L^{\infty}(X,\xi)\rtimes\Gamma\to\M$. Let $\tilde{\M}=\mathcal{M}'\cap\left(L^{\infty}(X,\xi)\rtimes\Gamma\right)$ be the relative commutant of $\M$ inside $L^{\infty}(X,\xi)\rtimes\Gamma$. Observe that $\tilde{\M}\subset L^{\infty}(X,\xi)'\cap\left(L^{\infty}(X,\xi)\rtimes\Gamma\right)$. Since the action $\Gamma\curvearrowright (X,\xi)$ is essentially free, the latter intersection coincides with $L^{\infty}(X,\xi)$. Hence, $\tilde{\M}\subset L^{\infty}(X,\xi)$ and therefore, $\mathbb{E}(\tilde{\M})=\tilde{\M}$. Now, since $(B,\nu)$ is an SH-space, $\mathbb{E}|_{\tilde{\M}}$ (in this case we view $\tilde{\M}$ as a subalgebra of $L^{\infty}(B,\nu)\rtimes\Gamma$) is either $\Gamma$-singular or $\mathbb{E}({\tilde{\M}})=\mathbb{C}$. We consider each of these cases one by one.  In the case when $\mathbb{E}(\tilde{\M})=\mathbb{C}$, since $\mathbb{E}(\tilde{\M})=\tilde{\M}$, we obtain that $\tilde{\M}=\mathbb{C}$. Let $\Lambda=\{g\in\Gamma: \lambda(g)\in \mathcal{M}\}$. Since $\mathcal{M}$ is $\Gamma$-invariant, it is easy to see that $\Lambda\triangleleft\Gamma$. Moreover, it is clear from the construction that $L^{\infty}(X,\xi)\rtimes\Lambda\subseteq\mathcal{M}$. All that remains to show is that $\mathcal{M}\subset L^{\infty}(X,\xi)\rtimes\Lambda$. Since $\mathcal{M}$ is $\Gamma$-invariant, it follows from \thref{relationtocommutant} that \[\mathbb{E}_{\M}(\lambda(g))\lambda(g)^*\in\mathcal{M}'\cap\left(L^{\infty}(X,\xi)\rtimes\Gamma\right)=\mathbb{C},~\forall g\in \Gamma.\]
Therefore, we obtain that $\mathbb{E}_{\M}(\lambda(g))=a_g\lambda(g)$ for some $a_g\in\mathbb{C}$. Moreover, if $a_g\ne 0$, we see that $\mathbb{E}_{\M}(\lambda(g))\in L(\Lambda)$ just by construction. 
Let $\tau=\nu|_{L^{\infty}(X,\xi)}\circ\mathbb{E}\circ\mathbb{E}_{\mathcal{M}}$, and consider the $\|.\|_2$-norm on $L^{\infty}(X,\xi)\rtimes\Gamma$ associated with $\tau$. Now, for $x\in \mathcal{M}$ and an arbitrary $\epsilon>0$, we can find $f_1,f_2,\ldots, f_n\in L^{\infty}(X,\xi)$ and $s_1,s_2,\ldots, s_n\in \Gamma$ such that 
\[\left\|x-\sum_{i=1}^nf_i\lambda(s_i)\right\|_2<\epsilon.\]
Since $\mathbb{E}_{\M}|_{L^{\infty}(X,\xi)}=\text{id}$, it follows from remark~\ref{norminequality} that
\[\left\|\mathbb{E}_M(x)-\sum_{i=1}^nf_i\mathbb{E}_M(\lambda(s_i))\right\|_2<\epsilon.\] 
Moreover, since $a\in \mathcal{M}$ and $\mathbb{E}_{\M}|_{\mathcal{M}}=\text{id}$, we see that \[\left\|x-\sum_{i=1}^nf_i\mathbb{E}_{\mathcal{M}}(\lambda(s_i))\right\|_2<\epsilon.\] 
Let us now observe that $\mathbb{E}_{\mathcal{M}}(\lambda(s_i))\in L(\Lambda)$ for each $i=1,2,\ldots,n$. As a consequence, we obtain that $\sum_{i=1}^nf_i\mathbb{E}_{\mathcal{M}}(\lambda(s_i))\in L^{\infty}(X,\xi)\rtimes\Lambda$. Since $\epsilon>0$ is arbitrary, it is evident that $x\in L^{\infty}(X,\xi)\rtimes\Lambda$. This finishes the proof for the case when $\nu\circ\mathbb{E}|_{\tilde{\M}}$ is invariant. 
If $\mathbb{E}|_{\tilde{\mathcal{M}}}$ is $\Gamma$-singular, it follows from \thref{singularlift} that $\mathcal{M}=\mathbb{E}(\mathcal{M})$. Since $L^{\infty}(X,\xi)\subset\mathcal{M}\subset L^{\infty}(X,\xi)\rtimes\Gamma$, it follows that $\mathbb{E}(\mathcal{M})=L^{\infty}(X,\xi)=\mathcal{M}$.
\end{proof}
\end{prop}
\starttocentries
\section{Towards the Conjecture}
Let $\Gamma$ be an irreducible lattice in a higher rank connected semisimple Lie group $G$ with a trivial center and no non-trivial compact factor, all of whose simple factors have real rank of at least two. It is known that $\Gamma$ admits a Furstenberg measure $\mu$, i.e., a random walk on $\Gamma$ such that the Furstenberg-Poisson boundary associated with a random walk $\mu$ is realized on $G/P$. We denote by $\nu_P$ the corresponding Poisson measure. 

Let us now put \thref{comdessdprop} along with \cite[Corollary~F]{Houdayersurvey} in perspective. 
The first result gives us a description of the intermediate invariant subalgebras $\mathcal{M}$ of the form $L^{\infty}(G/Q,\nu_Q)\subset\mathcal{M}\subset L^{\infty}(G/Q,\nu_Q)\rtimes\Gamma$, where $P\le Q\lneq G$ is a closed subgroup. On the other hand, the second result gives a description of the intermediate algebras $\mathcal{M}$ with $L(\Gamma)\subset\mathcal{M}\subset L^{\infty}(G/P,\nu_P)\rtimes\Gamma$.  Observe that such a $\mathcal{M}$ is automatically $\Gamma$-invariant. At the same time, let us also observe that the invariant algebras $\mathcal{M}$ considered above either share the same group algebra part or the commutative algebra part with those of their upper and lower bounds.

Consequently, considering all of the above, we make the following conjecture.
\begin{conjecture*}
\thlabel{conj}
Let $\mathcal{M}$ be a $\Gamma$-invariant subalgebra of $L^{\infty}(G/P,\nu_P)\rtimes\Gamma$. Then, $\mathcal{M}$ is a crossed product of the form $L^{\infty}(G/Q,\nu_Q)\rtimes\Lambda$, where $\Lambda\triangleleft\Gamma$.
\end{conjecture*}
We can only address the above conjecture under a certain technical assumption. We briefly recall the notion of Poisson transform and some related properties for our later use. 
\begin{definition}
\thlabel{poissontransform}
Let $\mathcal{A}$ be a unital $\Gamma$-$C^*$-algebra and $\varphi$, a state on $\mathcal{A}$. The Poisson transform associated with $\varphi$ is the map $\mathcal{P}_{\varphi}:\mathcal{A}\to\ell^{\infty}(\Gamma)$ defined by $$\mathcal{P}_{\varphi}(a)(s)=\varphi(s^{-1}a),~a\in\mathcal{A},~s\in\Gamma.$$
\end{definition}
\noindent
\begin{remark}
\label{invariantelements}
 Let $\mathcal{A}^{\Gamma}$ denote the invariant elements in $\mathcal{A}$, i.e., $\mathcal{A}^{\Gamma}=\{a\in\mathcal{A}: s.a=a~\forall s\in\Gamma\}$. It is clear that $\mathcal{A}^{\Gamma}\subset\{a\in\mathcal{A}:\varphi(sa)=\varphi(a)~\forall s\in\Gamma\}$. We observe below that the other inclusion holds if $\mathcal{P}_{\varphi}$ is an isometry. Indeed, assume that this is the case. Let $a\in\mathcal{A}$ be such that $\varphi(sa)=\varphi(a)$ for all $s\in\Gamma$. Fix $g\in\Gamma$. Then, $\mathcal{P}_{\varphi}(a-ga)(s)=\varphi(s^{-1}a-s^{-1}ga)=0$ for all $s\in \Gamma$. Therefore, $\mathcal{P}_{\varphi}(a-ga)=0$. Since $\mathcal{P}_{\varphi}$ is an isometry, we see that $a-ga=0$ for all $g\in\Gamma$. Consequently, $a\in\mathcal{A}^{\Gamma}$.   
\end{remark}
\begin{remark} In the case that $\mathcal{A}$ is commutative, namely $\mathcal{A}=C(X)$, the Poisson transform $\mathcal{P}_{\varphi}$ is an isometry if and only if the measure $\varphi$ is contractible in the sense of Azencott (\cite[Chapter-V, Proposition~2.1]{glasner1976proximal}).

If $\mathcal{A}=\mathcal{M}$ is a commutative von Neumann algebra, namely $\mathcal{M}=L^{\infty}(X,\varphi)$, then the Poisson transform is an isometry if and only if the measure $\varphi$ is SAT in the sense of Jaworski~\cite{jaworski1994strongly}.

For the relation between topological models of SAT-measures and contractible measures, see~\cite[Theorem~8.9]{furstenberg2010stationary}.
\end{remark}
We now briefly recall the notion of stationary states in the context of unital $C^*$-algebras and refer the readers to \cite{HartKal} for more details. 
\begin{definition}
Let $\mathcal{A}$ be a unital $\Gamma$-$C^*$-algebra. Let $\mu\in\text{Prob}(\Gamma)$. A state $\varphi\in S(\mathcal{A})$ is called $\mu$-stationary if 
\[\mu*\varphi(a)=\sum_{s\in\Gamma}\mu(s)\varphi(s^{-1}a)=\varphi(a),~\forall a\in\mathcal{A}.\]
\end{definition}
For the canonical conditional expectation $\mathbb{E}:\mathcal{A}\rtimes_r\Gamma\to\mathcal{A}$, $\tau=\varphi\circ\mathbb{E}$ is a $\mu$-stationary state on $\mathcal{A}\rtimes_r\Gamma$ for any $\mu$-stationary state $\varphi$ on $\mathcal{A}$. Indeed, for any $a\in \mathcal{A}\rtimes_r\Gamma$, it follows from the $\Gamma$-equivariance of $\mathbb{E}$ that
\[\mu*\tau(a)=\sum_{s\in\Gamma}\mu(s)\tau(s^{-1}a)=\sum_{s\in\Gamma}\mu(s)\varphi(s^{-1}\mathbb{E}(a))=\varphi(\mathbb{E}(a))=\tau(a)\]
We now proceed to prove the following lemma. Unless otherwise stated, $\tau$ denotes $\nu_P\circ\mathbb{E}$. Here, $\mathbb{E}:L^{\infty}(G/P,\nu_P)\rtimes\Gamma\to L^{\infty}(G/P,\nu_P)$ is the canonical conditional expectation. We think of $\nu_P$ as a state on $L^{\infty}(G/P,\nu_P)$ given by the integration with respect to $\nu_P$. Recall that $(G/P,\nu_P)$ is the Furtsenberg-Poisson boundary associated with a random walk $\mu$ on $\Gamma$. In particular, $\nu_P$ is a $\mu$-stationary state on $L^{\infty}(G/P,\nu_P)$ and hence, it follows from the above observation that $\tau$ is a $\mu$-stationary state on $L^{\infty}(G/P,\nu_P)\rtimes\Gamma$.
\begin{lemma}
\thlabel{condinv}
Let $\mathcal{M}$ be a $\Gamma$-invariant subalgebra of the crossed product $L^{\infty}(G/P,\nu_P)\rtimes\Gamma$. Then, $\mathbb{E}\left(\mathcal{M}\right)\subset \mathcal{M}$.
\begin{proof}
Let us first consider the case when $\tau$ is $\Gamma$-invariant. For this case, since $\tau|_{\mathcal{M}}$ is $\Gamma$-invariant, we see that the restriction of $\nu_P$ on $\mathbb{E}\left(\mathcal{M}\right)$ is invariant. Since $(G/P,\nu_P)$ is the $(\Gamma,\mu)$-Furstenberg-Poisson boundary, the Poisson transform $\mathcal{P}_{\nu_{P}}:L^{\infty}(G/P,\nu_P)\to\ell^{\infty}(\Gamma)$ is an isometry (see~\cite[Proposition~2.2]{jaworski1994strongly}). Since $\nu_P$ is $\Gamma$-ergodic, using Remark~\ref{invariantelements}, we see that the only functions $f\in L^{\infty}(G/P,\nu_P)$ on which $\nu_P$ is invariant are the constant functions. Hence, $\mathbb{E}\left(\mathcal{M}\right)$ consists of constant functions only and therefore, $\mathbb{C}=\mathbb{E}(\mathcal{M})\subset\mathcal{M}$.\\ 
Now, assume that $\tau$ is not $\Gamma$-invariant. Let us observe that the action $\Gamma\curvearrowright\M$ is ergodic (cf.~\cite[Lemma~2.16]{KalPan}). Therefore, using \cite[Theorem~B]{BH19}, we see that there exists a closed subgroup $ P\le Q\lneq G$ and a $\Gamma$-equivariant von Neumann algebra embedding $\theta:L^{\infty}(G/Q) \to \mathcal{M}$, such that $\tau \circ \theta=\nu_{Q}$. Note that in this case, $\nu_{Q}$ is the push forward measure of $\nu_{P}$ under the canonical quotient map
$G/P \to G/Q$.  
Therefore, the composition $\mathbb{E} \circ \theta$ is a normal $\Gamma$-equivariant von Neumann
algebra homomorphism from $L^{\infty}(G/Q, \nu_{Q})$ into $L^{\infty}(G/P, \nu_{P})$. However, the
canonical embedding is the unique such map (\say{the uniqueness of the boundary map}, e.g., \cite[Theorem~2.14]{BadShal06}); hence, $\mathbb{E} \circ \theta = \text{id}|_{L^{\infty}(G/Q)}$. Since, $\mathbb{E}$ is a faithful conditional
expectation, it follows that $\theta = \text{id} |_{L^{\infty}(G/Q)}$ (see e.g., \cite[Lemma 3.3]{hamana1985injective}), and
as a consequence, we see that $ L^{\infty}(G/Q,\nu_{Q})\subset \mathcal{M} \subset L^{\infty}(G/P,\nu_P)\rtimes\Gamma$. Moreover, the action $\Gamma \curvearrowright (G/Q,\nu_{Q})$ is essentially free (see eg., \cite[Lemma~6.2]{BH19}). From the proof of \cite[Theorem~3.6]{Suz}, for each $a\in L^{\infty}(G/P,\nu_P)\rtimes_{\text{alg}}\Gamma$, we can find $p_1,p_2,\ldots,p_n\in L^{\infty}(G/Q,\nu_Q)$ such that $\sum_{i=1}^np_iap_i=\mathbb{E}(a)$. Hence, if $a\in \mathcal{M}\cap L^{\infty}(G/P,\nu_P)\rtimes_{\text{alg}}\Gamma$, then $\mathbb{E}(a)\in\mathcal{M}$. Now, a standard approximation argument yields $\mathbb{E}(\mathcal{M})\subset\mathcal{M}$.
\end{proof}
\end{lemma}
Our strategy is to consider the cases when $\mathbb{E}(\M)$ is trivial or not. In the situation where $\mathbb{E}(\mathcal{M})=\mathbb{C}$, we want to show that $\mathcal{M}\subset L(\Gamma)$ and use \cite[Theorem~1.1]{KalPan} to conclude that $\M=L(\Lambda)$ for some normal subgroup $\Lambda\triangleleft\Gamma$. 
In the other case, when $\mathbb{E}(\mathcal{M})$ is non-trivial, we claim that it is enough to show that $\mathcal{M}\subset \mathbb{E}(\M)\rtimes\Gamma$. 

Indeed, let us look at $\mathcal{M}\cap L(\Gamma)$ which is a $\Gamma$-invariant subalgebra of $L(\Gamma)$ and hence by \cite[Theorem~1.1]{KalPan}, is of the form $L(\Lambda)$ for some normal subgroup $\Lambda\triangleleft\Gamma$. From \thref{condinv}, we already know that $\mathbb{E}(\mathcal{M})\subset\mathcal{M}$ and hence, $\mathbb{E}(\mathcal{M})\rtimes\Lambda\subset \mathcal{M}$. Whenever $\mathcal{M}\subseteq\mathbb{E}(\M)\rtimes\Gamma$, we obtain that
\[\mathbb{E}(\mathcal{M})\rtimes\Lambda\subset\mathcal{M}\subset\mathbb{E}(\M)\rtimes\Gamma,\]and then, \thref{conj} would be a consequence of \thref{comdessdprop}. We give an abstract condition that makes use of the tightness property of the $\mu$-boundaries \cite{hartman2021tight} which forces $\mathcal{M}\subset\mathbb{E}(\mathcal{M})\rtimes\Gamma$. For a $\Gamma$-invariant subalgebra $\mathcal{M}$, we denote by $\langle \mathcal{M}, L(\Gamma)\rangle$, the von Neumann algebra generated by $\mathcal{M}$ and $L(\Gamma)$.
\begin{prop}
\thlabel{upperboundforM}
Let $\mathcal{M}$ be a $\Gamma$-invariant subalgebra of $L^{\infty}(G/P,\nu_P)\rtimes\Gamma$. Suppose there exists a $\Gamma$-equivariant normal ucp map $\Phi: \left\langle\mathcal{M},L(\Gamma)\right\rangle\to\mathbb{E}(\mathcal{M})\rtimes\Gamma$. Then, $\mathcal{M}\subset \mathbb{E}(\mathcal{M})\rtimes\Gamma$.
\begin{proof}
It follows from \cite[Corollary~F]{Houdayersurvey} that $\left\langle\mathcal{M},L(\Gamma)\right\rangle$ is of the form $L^{\infty}(G/Q,\nu_Q)\rtimes\Gamma$, where $P\le Q\lneq G$ is a closed subgroup. We claim that 
\begin{align}
\label{topbottomequal}
L^{\infty}(G/Q,\nu_Q)\rtimes\Gamma=\mathbb{E}(\M)\rtimes\Gamma\end{align} which in turn will imply that $\mathcal{M}\subset \mathbb{E}(\M)\rtimes\Gamma$.

In order to show this, we first argue that $C(G/Q)\xhookrightarrow{}L^{\infty}(G/Q,\nu_Q)$ is $\Gamma$-tight. This follows from the fact that the compact space $G/Q$ has a unique $\mu$-stationary measure $\nu_Q$. Therefore, using \cite[Corollary~2.13]{hartman2021tight} we see that $C(G/Q)\xhookrightarrow{}L^{\infty}(G/Q,\nu_Q)\rtimes\Gamma$ is $\Gamma$-tight. Since $L^{\infty}(G/Q,\nu_Q)\rtimes\Gamma$ is generated as a von Neumann algebra by $C(G/Q)$ and $L(\Gamma)$, we see that the inclusion $L(\Gamma)\xhookrightarrow{} L^{\infty}(G/Q,\nu_Q)\rtimes\Gamma$ is co-tight in the sense of \cite[Definition~4.1]{hartman2021tight}. In particular, the inclusion $\mathbb{E}(\M)\rtimes\Gamma\xhookrightarrow{}L^{\infty}(G/Q,\nu_Q)\rtimes\Gamma$ is co-tight. By our assumption, $\Phi:L^{\infty}(G/Q,\nu_Q)\rtimes\Gamma\to\mathbb{E}(\M)\rtimes\Gamma$ is a $\Gamma$-equivariant conditional expectation.  Equation~(\ref{topbottomequal}) is now a consequence of \cite[Lemma~4.5]{hartman2021tight}, where $\text{C}=\mathbb{E}(\mathcal{M})\rtimes\Gamma$ and $\text{B}=L^{\infty}(G/Q,\nu_Q)\rtimes\Gamma$.   
 \end{proof}
\end{prop}
\begin{remark}
Let us note that if $\mathcal{M}\subset L^{\infty}(G/P,\nu_P)\rtimes\Gamma$ is a $\Gamma$-invariant subalgebra, then $\mathbb{E}(\M)\subset \M$ (cf. \thref{condinv}).
Moreover, it is enough to construct a $\Gamma$-equivariant normal ucp map $\Psi:\langle \M, L(\Gamma)\rangle\to \M$. Indeed, if such a $\Psi$ exists, then $\Phi$ (of~\thref{upperboundforM}) can be constructed by composing $\Psi$ with the canonical conditional expectation $\mathbb{E}$, i.e., $\Phi=\mathbb{E}\circ\Psi$. We also note that $\M\cap L(\Gamma)=L(\Lambda)$, where $\Lambda\triangleleft\Gamma$. There exists a normal faithful conditional expectation  $\tilde{\mathbb{E}}_{\Lambda}:L^{\infty}(G/P,\nu_P)\rtimes\Gamma\to L^{\infty}(G/P,\nu_P)\rtimes\Lambda$ defined by
\[\tilde{\mathbb{E}}_{\Lambda}\left(f_g\lambda_g\right)=\left\{ \begin{array}{ll}
0 & \mbox{if $g\not\in \Lambda$}\\f_g\lambda(g) & \mbox{otherwise}\end{array}\right\}\]
We refer the reader to \cite[Proposition~2]{choda1978galois} for proof of the above. Since $\mathbb{E}(\M)\subset\M$ and $\M\cap L(\Gamma)=L(\Lambda)$, it follows that $\tilde{\mathbb{E}}_{\Lambda}(\M)\subset\M$. Therefore, to prove the conjecture, it is enough to show that $\tilde{\mathbb{E}}_{\Lambda}\left(\langle \M, L(\Gamma)\rangle\right)\subset\M$. Once this is established, then $\Phi=\mathbb{E}\circ\tilde{\mathbb{E}}_{\Lambda}$ will be the required map. However, we do not know how to show this.
\end{remark}
\bibliographystyle{amsalpha}
\bibliography{name}
\end{document}